\documentclass{amsart}

\usepackage{geometry,graphicx,amssymb,amsmath,amsbsy,eucal,amsfonts,mathrsfs,amscd,bm,tcolorbox, subcaption, enumitem, tikz-cd}
\usepackage[hidelinks]{hyperref}
\usepackage[all]{xy}

\numberwithin{equation}{section}
\allowdisplaybreaks[1]

\newtheorem{theorem}{Theorem}[section]
\newtheorem{lemma}[theorem]{Lemma}

\newtheorem{proposition}[theorem]{Proposition}
\theoremstyle{definition}
\newtheorem{definition}[theorem]{Definition}
\theoremstyle{remark}

\newtcolorbox{strangbox}{
  title=\textbf{Strang's Conjecture (without singular vertices)}
}
\newtcolorbox{strangboxx}{
  title=\textbf{Strang's Conjecture (with singular vertices)}
}
\usepackage{mydef}

\newcommand{\sk}{S_h^k}
\newcommand{\vk}{\mathbf{V}_h^{k-1}}
\newcommand{\vone}{\mathbf{V}_h^{1}}
\newcommand{\dzero}{D_h^{0}}
\newcommand{\dk}{D_h^{k-2}}
\newcommand{\qk}{Q_h^{k-2}}
\newcommand{\qkt}{[Q_h^{k-2}]^\bot}
\newcommand{\Nov}{\mathring{N_v}}
\newcommand{\Noe}{\mathring{N_e}}

\begin{document}
\title[]{Strang's conjecture: Positive result on Strong Collapsible complexes and a code to check the conjecture}
\author{Aranzazu Romero}%
 \address{Mathematics, California State University-Long Beach, Long Beach, CA}%
 \email{aranzazu.romero01@student.csulb.edu}%
 \author{Renaldi-Bradley Bodombo }%
 \address{Mathematics, University of Maryland-College Park , College-Park, MD}%
\email{rbodombo@terpmail.umd.edu}%
% \author{\\ \textbf{Advisors: }}
% \author{Johnny Guzm\`an}
% \author{Maurice Fabien}
% \author{Elizabeth Rubio}
% \author{Parneet Gill}

 \makeatletter
\@namedef{subjclassname@2020}{\textup{2020} Mathematics Subject Classification}
\makeatother

%\date{\today}
%\textbf{Advisors:} Johnny Guzm\`an, Maurice Fabien, Elizabeth Rubio, Parneet Gill 
\begin{abstract}
Splines, which are piecewise polynomial functions with given smoothness, are used in numerous applications such as computer aided geometric design, algebraic geometry and for implementation of the finite element method. In many applications, it is important to know the dimension of the space of splines over a domain, but this continues to be an open problem. Strang's conjecture predicts the dimension of the space of bivariate $C^1$ splines of degree at most k over a polygonal domain with no holes. For the cases,  $k=4$ and $k \geq 5$ the conjecture has been proven true. For the case $k=2$, John Morgan and L. Ridgway Scott provided a counterexample. In this project, we begin by relating the space of bivariate $C^1$ splines of degree at most $k$ to various other function spaces. With the aid of some intrinsic operators on these spaces, we obtain a complex that yields an equivalence statement to Strang’s conjecture. The equivalent statement admits a computational approach to the conjecture and thus we present an NGSolve code to check Strang’s conjecture for a given mesh. In addition, we prove that if the simplicial complex associated to a triangulation is strongly collapsible (with a non-collinear condition), then Strang's conjecture holds for $k=2$.\\ 

\textbf{Advisors: Johnny Guzm\'an, Maurice Fabien, Elizabeth Rubio, and Parneet Gill} 
\end{abstract}

\maketitle
% \tableofcontents

\section{Introduction}

Let $\Omega \subset \mathbb{R}^2 $ be a connected two dimensional polygonal region with no holes. Let $\Th$ be a triangulation of $\Omega$ and define the space $S_h^k$ to be the space of continuously differentiable bivariate splines of order $k$ on the triangulation. 
\begin{align*}
    S_h^k(\Th)= \{ v \in C^1(\Omega): v|_T \in P^k(T), \forall T \in \Th \}.
\end{align*}
where $P^k(T)$ is the space of all bivariate polynomials of order at most $k$ on a triangle $T$ of a triangulation.
We will consider the following number associated to the triangulation $\Th$ which we will refer to as Strang's Number: 
\begin{equation}\label{eq:strangdim}
    \mathcal{S}_h^k= \binom{k+2}{2}N_T-(2k+1)\mathring{N_e} +3\mathring{N_v}  
\end{equation}
where $N_T, N_e, N_v$ are the number of triangles, edges and vertices of the triangulation, respectively. The number of interior edges and interior vertices are denoted by $\mathring{N}_e, \mathring{N}_v$, respectively.\\
\\
\noindent It was conjectured by Strang \cite{strang1973piecewise} that the dimension of $\sk(\Th)$ is given by \eqref{eq:strangdim}.

\begin{strangbox}

    \begin{equation*}
     \dim S_h^k(\Th)= \mathcal{S}_h^k 
     %\binom{k+2}{2}N_T-(2k+1)\mathring{N_e} +3\mathring{N_v}   
    \end{equation*}
\end{strangbox}

\noindent It was later shown that one should alter the right-hand side to take into account singular vertices (see Figure \ref{fig:singulavert}). These are vertices on the mesh that have four edges emanating from them such that these four edges lie on two lines. 

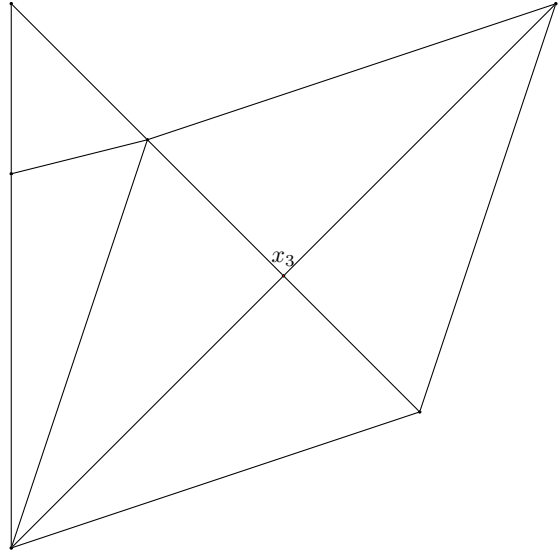
\begin{figure}[htbp]
    \centering
    \begin{subfigure}[b]{0.45\textwidth}
        %\centering
        \begin{tikzpicture}[scale=0.9, transform shape]
            \node[circle, draw, fill=black, inner sep = 0pt] (A) at (0,4){};
            \node[circle, draw, fill=black, inner sep = 0pt] (B) at (4,0){};
            \node[circle, draw, fill=black, inner sep = 0pt] (C) at (-2,-2){};
            \node[circle, draw, fill=black, inner sep = 0pt] (D) at (-2,3.5){};
            \node[circle, draw, fill=black, inner sep = 0pt] (E) at (-2,6){};
            \node[circle, draw, fill=black, inner sep = 0pt] (G) at (6,6){};
            \node[circle, draw, fill=red, label=above:{$x_3$}, inner sep = 0pt] (S) at (2,2){};
        
            %Remove all the edges necessary to satisfy the definition of Beta(E,K_2) Using previous complex
            
            \draw (A) -- (D);
            \draw (A) -- (C);
            \draw (D) -- (C);

            \draw (C) -- (S);
            \draw (A) -- (S);

            \draw (C) -- (B);
            \draw (B) -- (S);

            \draw (A) -- (G);

            \draw (G) -- (S);

            \draw (G) -- (B);

            \draw (A) -- (E);
            \draw (E) -- (D);
            
        \end{tikzpicture}
        
    \end{subfigure}
    \caption{Example of a singular vertex, $x_3$.}
    \label{fig:singulavert}
\end{figure}
The conjecture should read
\begin{strangboxx}
        \begin{equation*}
       \dim\mathcal{S}_h^k= \binom{k+2}{2}N_T-(2k+1)\mathring{N_e} +3\mathring{N_v} + \sigma  
    \end{equation*}
\end{strangboxx}
where $\sigma$ represents the number of singular vertices\\\\

Scott and Morgan \cite{morgan1975nodal} showed that the conjecture is true for splines with polynomial order $k \ge 5$. Alfeld et al. \cite{alfeld1987explicit} showed the conjecture is true in the case of $k=4$. In the case $k=2$, Scott and Morgan provided a counter-example. The case $k=3$ remains open. It should also be mentioned that for generic meshes Billera \cite{billera1988homology} proved that Strang's conjecture holds. 

The goal of this research project is two fold. First we write a code to check if for a given mesh Strang's conjecture holds. Our approach is to (as has done before) exploit that the space of $C^1$ splines fit into a differential complex. Then checking Strang's conjecture boils down to building a matrix and checking the dimension of the null space. Our code is written using NGSolve. %removed the citation. 

The second goal is to prove that Strang's conjecture holds for $C^1$ piecewise quadratics ($k=2$) on strongly collapsible simplicial complexes (with a non-collinear condition that we explain below).   The Morgan-Scott mesh  \cite{morgan1990dimension} that is the counter example for Strang's conjecture for $k=2$ is collapsible to a point, but not strongly collapsible. Hence, this prompted us to see if we can prove a positive result for strongly collapsible meshes. We were able to do this but with an additional assumption. 

In this paper, Section \ref{sec:notation} defines some required notations and spaces to set up the problem and Section \ref{sec:diff_complex} introduces the known equivalence statement of the conjecture. With the context of the problem established, we present a code in Section \ref{sec:NGSolve_Code} that permits a computational approach for checking the validity of the conjecture on a given mesh. The methods of the code are explained here and the results are verified on known cases of the conjecture. Finally, in Section \ref{sec:strong_collapsing}, we introduce a class of strongly collapsing meshes in order present a proof of Strang's conjecture for a subset of the case $k=2$.

In this paper, Section \ref{sec:notation} defines some required notations and spaces to set up the problem and Section \ref{sec:diff_complex} introduces the known equivalence statement of the conjecture. With the context of the problem established, we present a code in Section \ref{sec:NGSolve_Code} that permits a computational approach for checking the validity of the conjecture on a given mesh. The methods of the code are explained here and the results are verified on known cases of the conjecture. Finally, in Section \ref{sec:strong_collapsing}, we introduce a class of strongly collapsing meshes in order present a proof of Strang's conjecture for a subset of the case $k=2$. 

% % % % % % %
% % % % % % %
% % % % % % %

\section{Preliminaries}\label{sec:notation}
We begin by defining the spaces that we work with and their dimensions.

Recall, $P^k(T)$ is the space of all bivariate polynomials of order at most $k$ on a triangle $T$ of a triangulation, and the dimension of $P^k(T)$ is commonly known to be 
\begin{equation*}
    \text{dim}(P^k(T))= \frac{(k+1)(k+2)}{2}
\end{equation*}
The space of piecewise continuous bivariate polynomials of degree at most $k-1$ is denoted
\begin{equation*}
    \vk:= \{v \in C(\overline{\Omega}) : v|_T \in P^{k-1}(T)\ \forall\ T \in \tau  _h\}.
\end{equation*}
The space of discontinuous bivariate polynomials of degree at most $k-2$ is denoted
\begin{equation*}
    \dk := \{ v : v|_T \in P^{k-2}(T) \quad \forall \quad T \in \tau _n \}. 
\end{equation*} 
The dimensions of these spaces are well known \cite{brenner2008mathematical} (and can easily be shown) to be 
\begin{subequations}
    \begin{alignat}{1}
    \dim \vk=& N_v + (k - 2)N_e + \frac{(k-3)(k-2)}{2}N_T\label{dimv}\\
    \dim \dk=&  \frac{k(k-1)}{2}N_T\label{dimd}
    \end{alignat}
\end{subequations}
    
% The space of most concern is \sk which is the space of all $C^1$ continuous piecewise splines over a domain $\Omega$. We define this space to be
% \begin{equation*}
%     \sk = \{ v \in C^1(\overline{\Omega}) : v|_t \}
% \end{equation*}
For brevity, we will use $\sk$ for   $\sk(\Th)$,  $V_h^{k-1}$  for $V_h^{k-1}(\Th)$, $\dk$  for $\dk(\Th)$, and $\vk$ for $\left[V_h^{k-1}\left(\Th\right)\right]^2$
With the information presented here and in the introduction, we can proceed. 

\section{A differential complex associated to $S_h^k(\Th)$}\label{sec:diff_complex}

% Moved to preliminaries

    %We have to consider two other spaces to do so
    
    % \begin{alignat}{1}
    % V_h^{k-1}(\Th) := & \{v\in C\left(\overline{\Omega}\right):v|_T\in P^{k-1}\left(T\right),\forall \text{ T}\in \Th \}\\
    % D_h^{k-2}(\Th) := & \{v:v|_T \in P^{k-2}\left(T\right),\forall\text{ T}\in\Th\}
    % \end{alignat}

    % Where $V_h^{k-1}\left(\Th\right)$ is the space of piecewise continuous bivariate polynomials of degree at most $k-1$, and $D_h^{k-2}$ is the space of discontinuous bivariate polynomials of degree at most $k-2$. For brevity, we will use $\sk$ for   $\sk(\Th)$,  $V_h^{k-1}$  for $V_h^{k-1}(\Th)$, $\dk$  for $\dk(\Th)$, and $\vk$ for $\left[V_h^{k-1}\left(\Th\right)\right]^2$
    
    % It is well known \cite{} (can easily be shown) that the dimensions of these spaces are given by the following. 
    % \begin{subequations}
    % \begin{alignat}{1}
    % \dim \vk=& N_v + (k - 2)N_e + \frac{(k-3)(k-2)}{2}N_T\label{dimv}\\
    % \dim \dk=&  \frac{k(k-1)}{2}N_T\label{dimd}
    % \end{alignat}
    % \end{subequations}

It has been known  that the space of splines $S_h^k(\Th)$, $\dk$, $\vk$ fit into a differential complex.
We can now write the complex as follows. \\
\begin{equation}
    S_h^k\xrightarrow{\quad\text{rot}\quad} \vk\xrightarrow{\quad\text{div}\quad} D^{k-2}_h.
\end{equation}

Where $\text{rot}$ is the rotational gradient and for a function $s$ is given by
\begin{equation}
    \rot s = \left[ {\begin{array}{c}
\partial_2 s \\ 
-\partial_1 s
\end{array} } \right] 
\end{equation}
where $\partial_j$ is the partial with respect to the $j$the variable. 
% And $\text{div}$ is the divergence operator, acting on vector valued functions in $\left[V_h^{k-1}\left(\Th\right)\right]^2$. If $\vec{v}\in\left[V_h^{k-1}\left(\Th\right)\right]^2$ then, we can write $\vec{v}$ as $\vec{v}(x_1,x_2)=\big(v_1(x_1,x_2),v_2(x_1,x_2)\big),\quad v_1,v_2 \in V_h^{k-1}\left(\Th\right)$. 

For a vector valued function $\vec{v}$  the divergence of $\vec{v}$ is given by 
\begin{equation}
\text{div } \vec{v}=\partial_{1}v_1+\partial_{2}v_2.
\end{equation}

%Show the discrete de Rham sequence given is a complex.
We will use the following result. 
\begin{proposition}
    Range$(\rot, \sk) = \text{ker}(\dive, \vk)$.
\end{proposition}
\begin{proof}
    It follows from the definitions of the rotational gradient and divergence operators that $\dive(\rot s) = 0 $ for all $s \in \sk$.  Moreover, clearly the $\rot$ maps a function in $\sk$ to a function in $\vk$.  Hence, Range$(\rot, \sk)  \subset \text{ker}(\dive, \vk)$.\\

    To prove the reverse inclusion we let $v \in  \text{Ker }(\dive, \vk)$ and note that it is well known that there exists a $s \in H^1(\Omega)$  such that $\rot  s =v$. However, we need to argue that $s \in \sk$.  On each triangle $T \in \Th$ we have that $\rot s=v$ and since the entries of $v$ are in $P^{k-1}(T)$  it can be shown that $s|_T \in P^k(T)$. Since $s \in H^1(\Omega)$  this implies that $s$ is continuous.  Using that $\rot s=v$ is continuous we have have that $s$ is $C^1$ which implies that $s \in \sk$. 
\end{proof}

We can now prove an equivalent statement for Strang's conjecture. To do so, we use the following lemma.

\begin{lemma} Assume the mesh/triangulation $\Th$ has no singular vertices then  following relationship holds 
\begin{equation}
  \mathcal{S}_h^k=\dim \vk - \dim \dk+1. 
\end{equation}
\label{SwithVandD}
\end{lemma}
\begin{proof}
Let $\mathcal{D}_h^k$ be the right hand side: 

\begin{equation*}
  \mathcal{D}_h^k  :=\dim \vk - \dim \dk+1.
\end{equation*}
Using \eqref{dimv} and \eqref{dimd} we have 
\begin{equation*}
    \mathcal{D}_h^k=N_v + (k-2)N_e+\frac{(k-3)(k-2)}{2}N_T - \left(\frac{k(k-1)}{2}N_T\right)+1
\end{equation*}

We will use the following identities which are well known. 
\begin{subequations}
\begin{alignat}{1}
    -1 + N_v - N_e + N_T =&0 \label{id1}\\
    N_v^{\partial}=& N_e^{\partial} \label{id2} \\
    2\mathring{N_e}+N_e^{\partial} = &3N_T \label{id3}\\
    \mathring{N_v}-\mathring{N_e}+N_T-1 = & 0 \label{id4}\\
    N_v =& \mathring{N_v}+N_v^{\partial} \label{id5}\\
    N_e =& \mathring{N_e} + N_e^{\partial} \label{id6}
\end{alignat}
\end{subequations}
Here, $N_v^\partial$(respectively, $N_e^\partial$) represents boundary vertices (respectively boundary edges.) The following is the computation that demonstrates that the desired result holds. 
\begin{alignat*}{2}
\mathcal{S}_h^k-\mathcal{D}_h^k &= \binom{k+2}{2}N_T-(2k+1)\mathring{N_e} +3\mathring{N_v}-\left(N_v + (k-2)N_e+\frac{(k-3)(k-2)}{2}N_T - \left(\frac{k(k-1)}{2}N_T\right)+1\right) \qquad &&
\end{alignat*}
\\

We first group terms with number of edges, number of vertices, and number of triangles together, and simplify the coefficient of the number of triangles. Doing so, we obtain
\begin{alignat*}{2}
\mathcal{S}_h^k\mathcal{-D}_h^k&=3N_v^{o}-2N_v-(2k+1)\Noe-2(k-2)N_e + \frac{k^2+3k +2}{2}N_T-\frac{k^2-9k+12}{2}N_T-1 \qquad && \\ 
        &=3\Nov-2N_v-(2k+1)\Noe-2(k-2)N_e+\left(\frac{12k-10}{2}\right)N_T-1 \qquad && \\
        &= 3N_v^{o}-2N_v-(2k+1)\Noe-2(k-2)N_e+(6k-5)N_T\\
\end{alignat*}

We now write out all edges (respectively vertices) in terms of boundary edges (respectively vertices) and interior edges (respectively vertices) to obtain
\begin{alignat*}{2}
        \mathcal{S}_h^k\mathcal{-D}_h^k&=3\Nov-2\left(\Nov+N_v^{\partial}\right)-(2k+1)\Noe-2(k-2)N_e+(6k-5)N_T -1\qquad &&  \text{by } \eqref{id5} \\
        &= 3\Nov-2(\Nov+N_e^{\partial}) - (2k+1)\Noe-2(k-2)N_e+(6k-5)N_T-1 \qquad && \text{by }\eqref{id2}\\
        &= \Nov-2N_e^{\partial}-(2k+1)\Noe+(4-2k)(\Noe+N_e^{\partial}) + (6k-5)N_T-1 \qquad &&  \text{by } \eqref{id6}\\
\end{alignat*}

We now distribute the terms accordingly and we recognize one of the identities by moving some terms around. Doing so, we see
\begin{alignat*}{2}
      \mathcal{S}_h^k\mathcal{-D}_h^k  &= \Nov-2N_e^{\partial}-2k\Noe-\Noe+4\Noe-2k\Noe+4N_e^{\partial}-2kN_e^{\partial}+6kN_T-5N_T-1 \qquad &&  \\
        &= \Nov-\Noe-1 -5N_T -2N_e^{\partial}-2k\Noe+4N_e^{\partial}-2kN_e^{\partial}+4\Noe - 2k\Noe+6kN_T \qquad &&  \\
        &= \left(\Nov-\Noe-1+N_T\right) -6N_T - 2N_e^{\partial} - 2k\Noe + 4N_e^{\partial} - 2kN_e^{\partial} + 4\Noe - 2kN_e^{\partial}+6kN_T \qquad &&  \\
        &= (0) -6N_T - 2N_e^{\partial} - 2k\Noe+4N_e^{\partial}-2kN_e^{\partial}+4\Noe-2k\Noe+6kN_T \qquad &&   \text{by }\eqref{id4}\\\\
\end{alignat*}

Now, we group like terms and use some identities to substitute all boundary edges. We obtain the following:
\begin{alignat*}{2}
        \mathcal{S}_h^k\mathcal{-D}_h^k&= 6kN_T-6N_T+2N_e^{\partial} -2kN_e^{\partial}-4k\Noe + 4\Noe\\
        &= 6kN_T-6N_T+2(N_e-\Noe)-2k(N_e-\Noe)-4k\Noe+4\Noe \qquad &&   \text{by }\eqref{id6}\\
        &=6kN_T-6N_T + 2N_e-2\Noe-2kN_e+2k\Noe-4k\Noe+4\Noe \qquad &&  \\
        &= 6kN_T-6N_T+2N_e-2kN_e+2\Noe-2k\Noe \qquad &&  \\
\end{alignat*}

We now notice a common factor, $2-2k$, for each term. Upon factorization, we manipulate some terms to obtain another identity which will be easy to recognize.
\begin{alignat*}{2}
      \mathcal{S_h^k-D_h^k}  &= -3N_T(2-2k)+N_e(2-2k) + \Noe(2-2k) \qquad &&  \\
        &= (2-2k)\left[-3N_T+N_e+\Noe\right] \qquad &&  \\
        &= (2-2k)\left[-3N_T + (\Noe+N_e^{\partial})+\left(2\Noe-\Noe\right)\right] \qquad &&   \text{by } \eqref{id6}\\
        &= (2-2k)\left[-3N_T+\Noe-\Noe+(N_e^{\partial}+2\Noe)\right] \qquad &&  \\
        &= (2-2k)\left[-3N_T+(N_e^{\partial}+2\Noe)\right] \qquad &&  \\
        &= (2-2k)\left[-3N_T+3N_T\right] \qquad &&  \text{by } \eqref{id3}\\
        &= (2-2k)(0) \qquad &&  \\
        &= 0 \qquad &&  
\end{alignat*}
\end{proof}

Now we can state the equivalent formulation of Strang's conjecture. 
\begin{theorem} Assume that $\Th$ has no singular vertices.   Strang's conjecture holds on $\Th$  (i.e. $\dim(\sk) = \mathcal{S}_h^k$)  if and only if 
\begin{equation}
  \text{Range} \left(\dive, \vk\right) =\dk.\label{RangeEquiv}  
\end{equation}
\end{theorem}

\begin{proof}
    $(\Rightarrow)$\\
    Assume that $\text{Range}\left(\dive, \vk\right)=D_h^{k-2}$. We want to show that Strang's conjecture is true. To see this, we will employ \textbf{\eqref{SwithVandD}}, and make use of the rank nullity theorem.

    Of course, it is important to verify that using the rank-nullity theorem is valid. Indeed, \textbf{rot}, and \textbf{div} are differential operators, so they are linear operators. As such, the linear maps in the complex must also satisfy the rank nullity theorem which states that for a linear map, $T,$ and finite vector spaces $A$ and $B$ that satisfy
\[T:A\longrightarrow B\]
\begin{equation}\text{dim}(A) = \text{rank}(T)+\text{nullity}(T).\label{rank-Nullity}\end{equation}
where we recall that  
\begin{alignat*}{1}
\text{rank}(T)= &\dim \text{Range}(T, A),\\
\text{nullity}(T)=&  \dim \text{Ker}(T, A).
\end{alignat*}

We can now write the following:
\begin{align*}&\text{rot}:\sk \longrightarrow \vk \\
&\text{div}: \vk \longrightarrow D_h^{k-2}\end{align*}\\
It follows from \eqref{rank-Nullity}, that
\begin{align}
    \text{dim}\left(S_h^k\right)&=\text{dim}\left(Range\left(\text{rot},S_h^k\right)\right)+\text{dim}\left(ker\left(rot,S_h^k\right)\right) \label{dimSRankNull}\\
    \text{dim}\left(\vk\right)&=\text{dim}\left(Range\left(div,\vk\right)\right) + \text{dim}\left(ker\left(div,\vk\right)\right)\label{dimVRankNull}
\end{align}

By hypothesis, we can see that 
\[\text{dim}\left(\vk\right)=\text{dim}\left(D_h^{k-2}\right)+\text{dim}\left(ker\left(div,\vk\right)\right)\]

But, we can use the fact that our sequence is exact to write
\[\text{dim}\left(\vk\right)=\text{dim}\left(D_h^{k-2}\right)+\text{dim}\left(Range\left(\text{rot},S_h^k\right))\right)\]
We can now substitute \eqref{dimSRankNull} into the equation above to see that
\[\text{dim}\left(\vk\right)=\text{dim}\left(D_h^{k-2}\right)+\text{dim}\left(S_h^k\right)-\text{dim}\left(\left(ker(\text{rot}, S_h^k\right)\right)\]

Since $\text{ker}\left(\text{rot},S_h^k\right)$ is the space of constant functions we have:

\[\text{dim}\left(ker\left(\text{rot},S_h^k\right)\right)=1.\]

Using this information, we can see that the result holds by writing
\[\text{dim}\left(\vk\right)=\text{dim}\left(D_h^{k-2}\right)+\text{dim}\left(S_h^k\right)-\text{dim}\left(\left(ker(\text{rot}, S_h^k\right)\right)\]
That implies that 
\[\text{dim}\left(\vk\right)=\text{dim}\left(D_h^{k-2}\right)+\text{dim}\left(S_h^k\right)-1\]
Which, in turn, means that
\[\text{dim}\left(\vk\right)-\text{dim}\left(D_h^{k-2}\right)+1=\text{dim}\left(S_h^k\right)\]

    So, we obtain
    \begin{math}\mathcal{S}_h^k= \text{dim}\left(S_h^k\right),\end{math} by \text{Lemma \eqref{SwithVandD}}\\\\

    $(\Leftarrow)$\\
    Assume $\mathcal{S}_h^k=\dim\left(\sk\right).$ We want to show $\text{Range}\left(\dive , \vk\right)=D_h^{k-2}.$  We will complete this proof by obtaining a contradiction. So, assume, to get a contradiction, that $\text{Range}\left(\dive , \vk\right)\subsetneq D_h^{k-2}$\\\\

    Since these are finite vector spaces, it must be true that \begin{equation}\dim\left(\text{Range}\left(\dive , \vk\right)\right)<\dim\left(D_h^{k-2}\right) \label{inequality1}. \end{equation}

    Notice that if the inequality was not strict, then $\text{Range}\left(\dive , \vk\right)$ would span $\dk$, since they are vector spaces. As such, we can use the inequality and the equations \eqref{dimSRankNull} and \eqref{dimVRankNull}

    By substituting into \eqref{dimVRankNull} and using \eqref{inequality1} with the fact that the sequence is exact, we have
    \[\dim\left(\vk\right) < \dim\left(\dk\right)+\text{rank}\left(\dive, \vk\right).\]

    By substituting in \eqref{dimSRankNull}, we have
    \[\dim\left(\vk\right) < \dim\left(\dk\right)+\dim\left(\sk\right)-1.\]

    This means
    \[\dim\left(\vk\right) -\dim\left(\sk\right)+1 <\dim\left(\dk\right)\]

    But by \textbf{\eqref{SwithVandD}}, we get
    \[\dim\left(\vk\right) -[\dim\left(\vk\right)-\dim\left(\dk\right)+1]+1 <\dim\left(\dk\right)\]

    Working this inequality out, we get
    \begin{equation*}
        \dim\dk<\dim\dk
    \end{equation*}

    This is an absurdity, so, $\text{Range}\left(\dive, \vk\right)=\dk$
    
\end{proof}

\section{The NGSolve code to check the conjecture}\label{sec:NGSolve_Code}
In  this section, we develop notation to help us create a linear algebra problem to determine if Strang's conjecture holds for a given mesh. 
We begin by defining an inner product:
Suppose $p,q\in \dk$ we write
\begin{equation*}
    \langle p \ ,q\rangle\ := \ \int _\Omega p\cdot q.
\end{equation*} 
Now, we consider the inner product space: $(\qk,\ \langle \cdot \ ,\cdot\rangle )$, where $\qk \subset \dk$ and is given by
\begin{align*}
     \qk &:= \text{Range}(\dive, \sk) \\
     \intertext{and its orthogonal complement}
     \qkt &:= \{ p \in \dk \ :\ \langle p,\ q\rangle \ =\ 0 \text { for all } q \in \qk \}.
\end{align*}

We use the following result to develop a program using NGsolve which determines the rank and nullity associated with a given mesh.

\begin{theorem}\label{thm:equivthm}
Range $(\dive, \vk) = \dk$ if and only if dim $\qkt$ $= 0$.
\end{theorem}

\begin{proof}
  We have that  
  \begin{equation*}
  \dk= \qk\oplus \qkt.    
  \end{equation*}
 Hence, $\qk= \dk$ if and only if $\qkt$ is trivial.  
\end{proof}
% \begin{proof}
%     $(\Rightarrow)$
%     Assume Range $(\dive, \vk) = \dk$, then for all $p \in \dk$ there exists a $v \in \vk$ such that $\dive v = p$. Now, if 
%     \begin{equation}
%         \int_\Omega p\ \dive v = 0
%     \end{equation}
%     then 
%     \begin{equation}
%         \int_\Omega p^2 = 0
%     \end{equation}
%     implying that $p^2 =0$, hence $p=0$.

%     $(\Leftarrow)$
%     Assume, by way of contradiction, that Range $(\dive, \vk) \ne \dk$, that is Range $(\dive, \vk) \subsetneq  \dk$. We show that for all $ \ p \ \in \dk$, there exists a $v \in \vk2$ such that $\int _{\Omega} p\ \text{div}(v)\ \ne \ 0$.\\
%     Then $D_h^{k-2},\ < \cdot \ ,\cdot>$ is a Hilbert (inner product) space. 
%     Now, let $Q_h^{k-2} = \text{Range}(\dive, \vk )$. Since $Q_h^{k-2}\subsetneq \dk$, $\dk = \qk \oplus \qkt$, where $\qkt$ is the non-trivial orthogonal complement of Range $(\dive, \vk )$. Also, since $Q_h^{k-2}\subsetneq \dk$, for all $q \in Q_h^{k-2}$ there is a $v \in \vk$ so that $\dive v = q$\\
%     Now, choose $p \in [Q_h^{k-2}]^\bot$. Then, we have $\int _\Omega p\cdot q = 0$ for all $q \in Q_h^{k-2}$ since $<p,\ q> = 0.$ Which is a contradiction since we supposed the existence of at least one $v$ with div$(v) = q$ and $\int _{\Omega} p\ \text{div}v\ \ne \ 0$. 
% \end{proof}

In collaboration with Dr. Maurice Fabien, we implement a code using NGSolve to assemble and solve the matrix associated with a mesh and determine if Strang's conjecture holds. Our code takes in a mesh, and using the principles of Theorem \ref{thm:equivthm}, it figures out if dim $\qkt$ $= 0$ or not.  

The code uses a basis for our two spaces  $\dk= \text{span}\{p_1, \ldots, p_L\}$ and  $\vk= \text{span}\{v_1, \ldots v_M\}$. Then, it computes that matrix $M \times L$ matrix  $A$ with the following entries:
\begin{equation}
   A_{ji} =\langle \dive v_j,  p_i \rangle.
\end{equation}
Then, we find the $L$-vectors that are in the kernel of the matrix $A$:
\begin{align*}
  A\vec{X}=0.  
\end{align*}
Then, it is easy to see that 
\[
p:=X_1 p_1 + \cdots+X_L p_L, \quad \text{then }  p \in \qkt.
\]

Thus the  dim $\qkt$ $= 0$ if and only if the kernel of the matrix $A$ is trivial.  Or, Strang's conjecture holds for the given mesh if and only if Kernel of the matrix $A$ is trivial.

We begin by creating a helper function, \verb|assemble_strang_matrix(k, maxh, quad_order)| (ASM), where \emph{k} is the degree of the $C^1$ continuous piecewise polynomials, and \emph{maxh} is the maximum edge size for any edge in our triangulation. In ASM, we use another helper function that returns a mesh for ASM to create a matrix for. We define a vector-valued Sobolev ($H^1$) space to represent our $\vk$ space (\emph{V}), and a discontinuous $L^2$ space to represent out $\dk$ space (\emph{Q}). Then, create a product space of \emph{V} and \emph{Q} so we can solve the linear system $A\vec{x}=\vec{0}$, where $A$ is the matrix assembled from our mesh. Test and trial functions are used to create and assemble the bilinear form, or sparse system matrix, that is populated with integrators of the form 
\begin{equation}
    \int_\Omega p_i \dive v_j \quad \text{ for } i= 1,2,...,L \text{ and } j= 1, 2, ..., M 
\end{equation}
where $\{p_1,\ p_2,\ ...,\ p_L\}$ is a basis for $\dk$ and $\{v_1,\ v_2,\ ...,\ v_M\}$ is a basis for $\vk$. 

We extract the non-zero values of the sparse matrix and create a compressed sparse row matrix to create a matrix representation of our mesh. 

ASM is implemented through a procedural oriented program. We first choose \emph{k} (the degree of our $C^1$ continuous polynomials) and \emph{maxh}. Then, we call ASM which returns a SciPy sparse matrix (\emph{A}), the mesh used to create the matrix, \emph{V}, and \emph{Q}. \emph{A} is converted to a dense array representation (\emph{Ad}). We singular value decompose \emph{Ad} and sum over the on-zero singular values of \emph{A} to determine the rank. Using the Rank-nullity theorem of Linear Algebra, we can determine the nullity of our matrix. Then by \ref{thm:equivthm}, we can conclude that Strang's conjecture holds if our mesh yields a matrix with the trivial nullspace, and Strang's conjecture does not hold otherwise. 

\subsection{Verification of our code on the Morgan-Scott Mesh}
We verify our code using the Morgan-Scott mesh, a counterexample for Strang's conjecture in the case $k=2$. We begin by constructing the Morgan-Scott mesh in NGsolve (Figure 1) and create its matrix representation using ASM. 
As expected, we get a non-trivial nullity, $\text{nullity}(A) = 1$, implying Strang's conjecture doesn't hold for this mesh in the case $k=2$. 

\begin{figure}[h!]
    \centering
    \includegraphics[width=0.375\linewidth]{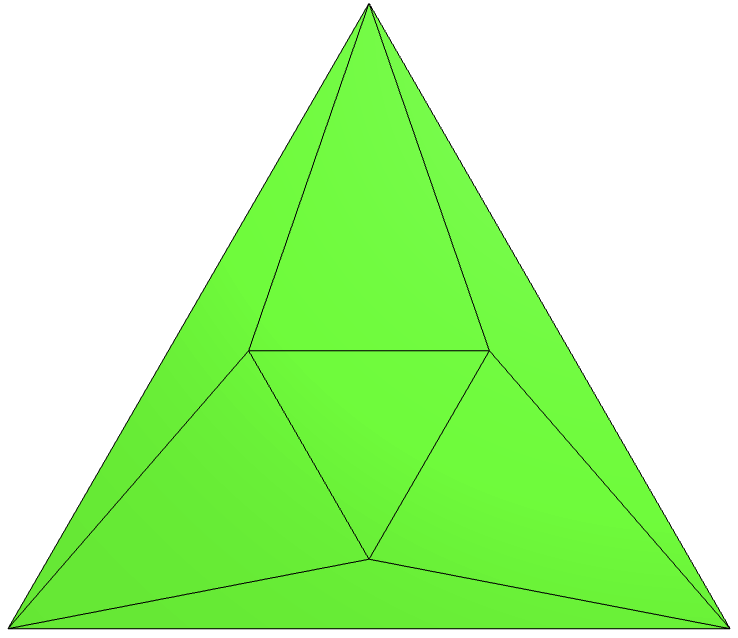}
    \caption{Morgan-Scott mesh created using NGsolve.}
    \label{fig:mesh1}
\end{figure}

After verifying our program works, we extract the nullspace basis for the Morgan-Scott mesh. The basis is: 
\begin{equation*}
    \vec{X}=[ 1,\ -2,\ 1,\ -2,\ 1,\ -2,\ 7]^T;
\end{equation*}
So, for any $t \in \mathbb{R}$, then
\begin{equation*}
    A (t \overrightarrow{X})= \overrightarrow{0}.
\end{equation*}

\section{Strong Collapsing Triangulation}\label{sec:strong_collapsing}
As we saw above (and was first verified by Morgan and Scott \cite{morgan1990dimension}) the simple Morgan-Scott mesh gives a counter-example in the case $k=2$. However, we also note that for a generic mesh the conjecture will hold.  We note that the simplicial complex associated with the Morgan-Scott mesh is a collapsible simplicial complex \cite{whitehead1950simple}. However,  as has been noted by others it is not strongly collapsible \cite{barmak2012strong} (see our definition below).

In this section we show that on any strongly collapsible complex of the plane that forms a triangulation will satisfy Strang's conjecture for $k=2$, provided that there are no quasi-singular vertices appearing in our process of a strong collapse (which we explain below).  We give a definition of a strong collapsible complex  that is suited for our case. For a planar simplicial complex that forms a triangulation, our definition will imply the notion in \cite{barmak2012strong}. However, the authors  in \cite{barmak2012strong} consider a much broader notion.

We will need a series of definitions.  As we have been doing, we assume that $K$ is a simplicial complex in the plane.
\begin{definition}
We say that the simplicial complex $K$ forms a triangulation if every edge or vertex of the mesh is the edge or vertex of at least one triangle of $K$.
\end{definition}
In other words, this does not allow edges or vertices that do not have at least one associated triangle. 
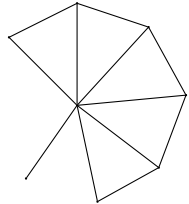
\begin{figure}[!h]
    \centering
        %\centering
        \begin{tikzpicture}[scale=0.45, transform shape]
            \node[circle, draw, fill=black, inner sep = 0pt] (A) at (0,4){};
            
            \node[circle, draw, fill=black, inner sep = 0pt] (B) at (-2,3){};
            
            \node[circle, draw, fill=black, inner sep = 0pt] (D) at (-1.5,-1.15){};
            \node[circle, draw, fill=black, inner sep = 0pt] (E) at (0.6,-1.83){};
            \node[circle, draw, fill=black, inner sep = 0pt] (F) at (2.4,-0.83){};
            \node[circle, draw, fill=black, inner sep = 0pt] (G) at (3.2,1.3){};
            \node[circle, draw, fill=red, inner sep = 0pt] (H) at (2.1,3.3){};
            \node[circle, draw, fill=red, inner sep = 0pt] (I) at (0,1){};
        
            %Remove all the edges necessary to satisfy the definition of Beta(E,K_2) Using previous complex
            
            \draw (A) -- (I);
            \draw (A) -- (B);
            \draw (B) -- (I);
            
            \draw (A) -- (H);
            \draw (H) -- (I);
            
            \draw (G) -- (I);
            \draw (G) -- (H);
            
            \draw (F) -- (I);
            \draw (F) -- (G);
        
            \draw (F) -- (E);
            \draw (E) -- (I);
        
            \draw (I) -- (D);
            
        \end{tikzpicture}

    \caption{This is not considered a triangulation.}
    \label{fig:hangingedge}
\end{figure}

\begin {definition}
    An edge $e \in \Delta_1(K)$ of the simplicial complex $K$ is a \emph{boundary edge} if  it is an edge of at most one triangle of $K$.
\end{definition}

\begin{definition}
 A vertex $v$ is a \emph{boundary vertex} of the simplicial complex $K$ if there  there is a boundary edge $e$ of K such that $v \in \Delta_0(e)$.
\end{definition}

% \begin{definition}
%     A vertex $v$ is an \emph{interior vertex} of a simplical complex $K$ if 
% \end{definition}
% \parneet{is boundary edge defined?}
% \begin{definition}
%     A \emph{boundary vertex} is a vertex that is not an interior vertex. 
% \end{definition}

We will use the following notation when $K$ is a simplex in the plane and $x \in \Delta_0(K)$ is a vertex.  
\begin{equation*}
  \mathfrak{Z}(x, K)=\{ \sigma \in \Delta(K): x \in \Delta_0(\sigma). \}  
\end{equation*}.
In other words $\mathfrak{Z}(x, K)$ are all the simplices of $K$ that have $x$ as a vertex. The union of these will be the star of $x$ in the simplicial complex $K$.

We also consider the collection simplicial complex  that make up the closed star. 
\begin{equation*}
  \overline{\mathfrak{Z}}(x, K)=\{ \sigma \in \Delta(K): \sigma \in  \Delta (T), \text{ for some } T \in \mathfrak{Z}(x, K) \}  
\end{equation*}.

\begin{figure}[h]
    \begin{subfigure}[b]{.45\textwidth}
        \centering
        \begin{tikzpicture}
            \node[circle, draw, fill=black, inner sep = 0pt] (A) at (0,4){};
            \node[circle, draw, fill=black, inner sep = 0pt] (B) at (-1,3){};
            \node[circle, draw, fill=black, inner sep = 0pt] (C) at (1,2){};
            \node[circle, draw, fill=black, inner sep = 0pt] (D) at (-3,3){};
            \node[circle, draw, fill=black, inner sep = 0pt] (E) at (-2,5){};
            \node[circle, draw, fill=black, inner sep = 0pt] (F) at (-2,1){};
            \node[circle, draw, fill=black, inner sep = 0pt] (G) at (1,1){};
            \node[circle, draw, fill=black, inner sep = 0pt] (H) at (-4,1){};
            \node[circle, draw, fill=black, inner sep = 0pt, label=above:{$x_0$}] (I) at (-4,5){};
    
            \filldraw [fill=black!20, draw=black, label={T1}] (I.center) -- (E.center) -- (D.center) -- cycle;
    
            \filldraw [fill=black!20, draw=black, label={T1}] (I.center) -- (H.center) -- (D.center) -- cycle;
    
            \node at (-2.75, 4.3) {$T_1$};
            \node at (-3.5, 3) {$T_2$};
            
            \draw [ultra thick] (E) -- (I) node[midway, above]{$e_1$};
            \draw (E) -- (D);
            \draw [ultra thick] (D) -- (I) node[pos=.3, right]{$e_2$};
            
            \draw (E) -- (B);
            \draw (B) -- (D);
    
            \draw (E) -- (A);
            \draw (B) -- (A);
        
            \draw (C) -- (D);
            \draw (B) -- (C);
    
            \draw (A) -- (C);
            
            \draw (C) -- (F);
            \draw (F) -- (D);
    
            \draw (C) -- (G);
            \draw (C) -- (F);
            \draw (F) -- (G);
    
            \draw (F) -- (H);
            \draw (H) -- (D);
    
            \draw [ultra thick](H) -- (I) node[midway, left]{$e_3$};
        \end{tikzpicture}
        \caption{The collection of simplicies that make up the star of $x_0$}
        \label{subfig:starx}
    \end{subfigure}
    \hfill
    \begin{subfigure}[b]{.45\textwidth}
        \centering
        \begin{tikzpicture}
            \node[circle, draw, fill=black, inner sep = 0pt] (A) at (0,4){};
            \node[circle, draw, fill=black, inner sep = 0pt] (B) at (-1,3){};
            \node[circle, draw, fill=black, inner sep = 0pt] (C) at (1,2){};
            \node[circle, draw, fill=black, inner sep = 0pt] (D) at (-3,3){};
            \node[circle, draw, fill=black, inner sep = 0pt] (E) at (-2,5){};
            \node[circle, draw, fill=black, inner sep = 0pt] (F) at (-2,1){};
            \node[circle, draw, fill=black, inner sep = 0pt] (G) at (1,1){};
            \node[circle, draw, fill=black, inner sep = 0pt] (H) at (-4,1){};
            \node[circle, draw, fill=black, inner sep = 0pt, label=above:{$x_0$}] (I) at (-4,5){};
    
            \filldraw [fill=black!20, draw=black, ultra thick] (I.center) -- (E.center) -- (D.center) -- cycle;
    
            \filldraw [fill=black!20, draw=black, ultra thick] (I.center) -- (H.center) -- (D.center) -- cycle;
    
            \node at (-3, 4.3) {$T_1$};
            \node at (-3.5, 3) {$T_2$};
            
            \draw [ultra thick] (E) -- (I) node[midway, above]{$e_1$};
            \draw (E) -- (D) node[midway, right]{$e_4$};
            \draw [ultra thick] (D) -- (I) node[pos=.3, right]{$e_2$};
            
            \draw (E) -- (B);
            \draw (B) -- (D);
    
            \draw (E) -- (A);
            \draw (B) -- (A);
        
            \draw (C) -- (D);
            \draw (B) -- (C);
    
            \draw (A) -- (C);
            
            \draw (C) -- (F);
            \draw (F) -- (D);
    
            \draw (C) -- (G);
            \draw (C) -- (F);
            \draw (F) -- (G);
    
            \draw (F) -- (H);
            \draw [ultra thick] (H) -- (D) node[midway, right]{$e_5$};
    
            \draw [ultra thick](H) -- (I) node[midway, left]{$e_3$};
    
        \end{tikzpicture}
        \caption{The collection of simplicies that make up the closed star of $x_0$.}
        \label{subfig:clstarx}
    \end{subfigure}
    \caption{Example of (a) $\mathfrak{Z}(x_0, K)$ and (b) $\overline{\mathfrak{Z}}(x_0, K).$}
        \label{fig:stars}
\end{figure}
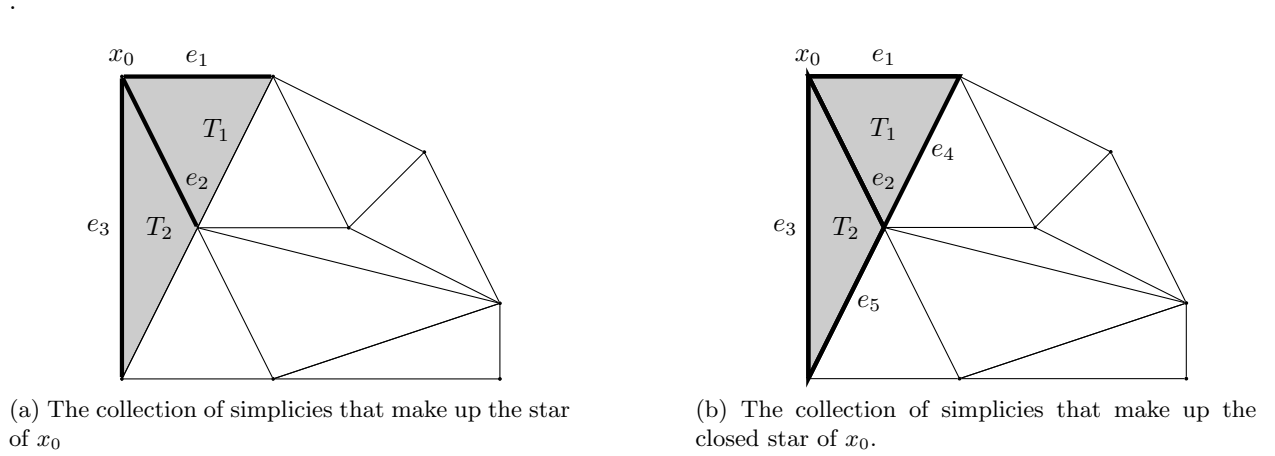 
We will also use the notation.
\begin{equation*}
\mathfrak{\beta}(x, K)= \mathfrak{Z}(x,K) \cup \{ \sigma  \in \overline{\mathfrak{Z}}(x, K) \backslash  \mathfrak{Z}(x, K): \sigma \notin \Delta(T),\forall T \in \Delta_2(K \backslash \mathfrak{Z}(x, K)) \}.    
\end{equation*}
These are all the simplices  of  $\mathfrak{Z}(x,K)$ plus those simplices belonging to $\overline{\mathfrak{Z}}(x, K)$ that are not sub-simplices of any other triangle not belonging to  $\mathfrak{Z}(x, K)$. 

We see that if $K$ is a simplical complex in the plane that make up a triangulation. And, if $x$ is a boundary vertex of $K$  
\begin{equation*}
   K \backslash \mathfrak{\beta}(x, K)  
\end{equation*}
is a simplicial complex that also forms a triangulation. 

\begin{definition}\label{strong}
    A simplicial complex \emph{K} of the plane that forms a triangulation is \emph{strongly collapsible} if there exists $x_0, x_1, ..., x_{N} \in \Delta_0(K)$ and sequence of simplicial complexes: 
    \begin{equation*}
        K=K_0 \supseteq K_1 \supseteq ... \supseteq K_N 
    \end{equation*} 
    such that \\
    1) $K_i$  forms a triangulation,  \\
    2) $x_i$ is boundary vertex of $K_i$ and contained in at most two triangles of $K_i$. If it is contained in two triangles $T_1$, and $T_2$ they share a common edge.  \\
    3) $K_{i+1}= \{ \sigma \in K_i: \sigma \notin \beta(x_i, K_i)\}$, \\ 
    4) $K_N$ is composed of one triangle and its sub-simplices.
\end{definition}

We immediately see that the Morgan-Scott mesh is not strongly collapsing since every boundary vertex has three triangles that contain it. Now, consider the following demonstration of strongly collapsing a simplicial complex, $K$.

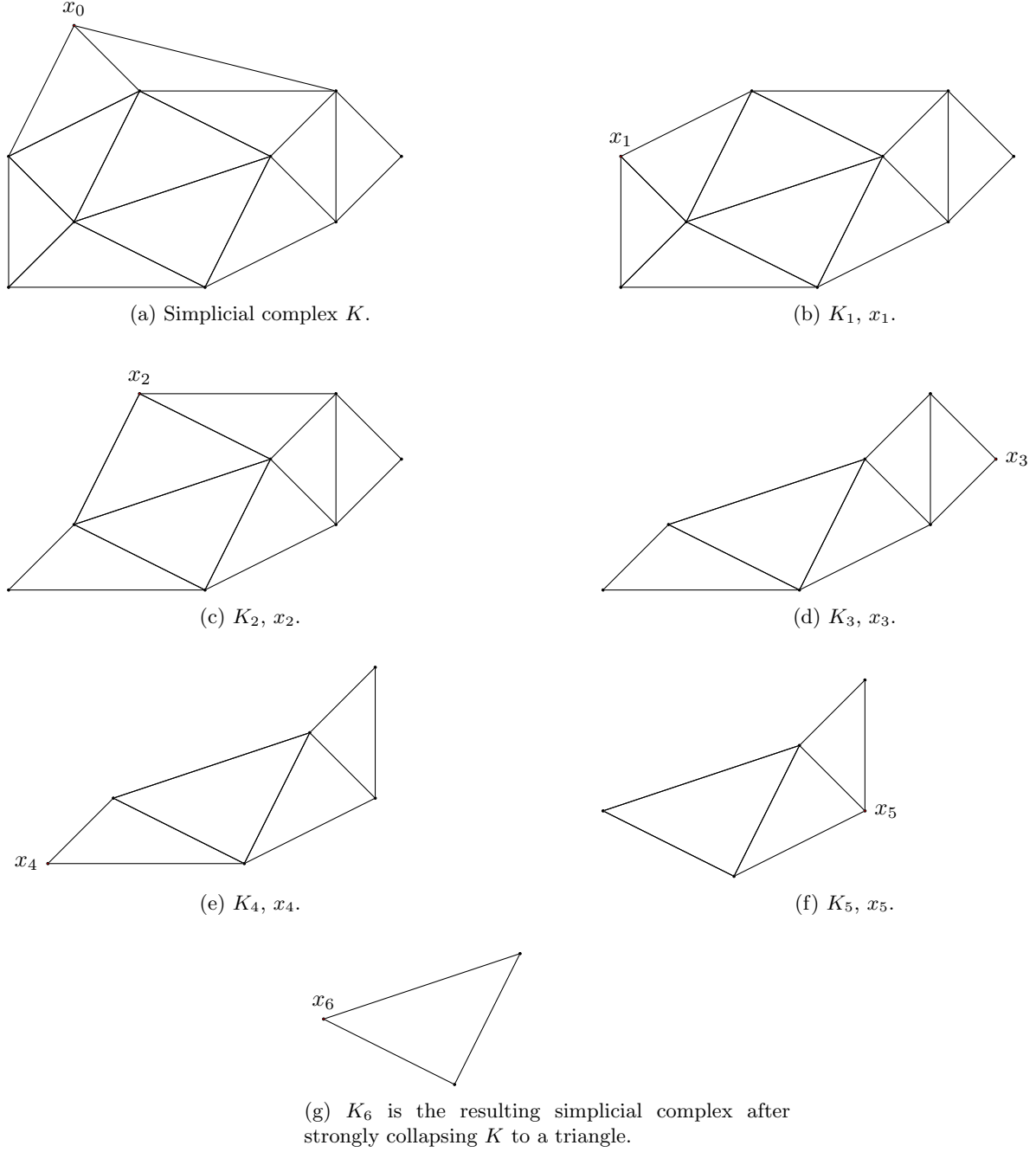
\begin{figure}[!htbp]
    \centering
    %K0
    \begin{subfigure}[b]{.45\textwidth}
        %\centering
        \begin{tikzpicture}
            % Define the vertices (nodes)
            \node[circle, draw, fill=black, inner sep = 0pt] (A) at (0,0){};
            \node[circle, draw, fill=black, inner sep = 0pt] (B) at (0,2){};
            \node[circle, draw, fill=black, inner sep = 0pt] (C) at (1,1){};
            \node[circle, draw, fill=red, label=above:{$x_0$}, inner sep = 0pt] (D) at (1,4){};
            \node[circle, draw, fill=black, inner sep = 0pt] (E) at (2,3){};
            \node[circle, draw, fill=black, inner sep = 0pt] (F) at (3,0){};
            \node[circle, draw, fill=black, inner sep = 0pt] (G) at (4,2){};
            \node[circle, draw, fill=black, inner sep = 0pt] (H) at (5,1){};
            \node[circle, draw, fill=black, inner sep = 0pt] (I) at (5,3){};
            \node[circle, draw, fill=black, inner sep = 0pt] (J) at (6,2){};
        
            % Draw the edge between them
            %Triangle by triangle
            \draw (A) -- (B);
            \draw (A) -- (C);
            \draw (B) -- (C);
        
            \draw (C) -- (E);
            \draw (B) -- (E);
            \draw (B) -- (C);
        
            \draw (B) -- (D);
            \draw (B) -- (E);
            \draw (D) -- (E);
        
            \draw (C) -- (E);
            \draw (C) -- (G);
            \draw (E) -- (G);
        
            \draw (E) -- (I);
            \draw (E) -- (G);
            \draw (G) -- (I);
        
            \draw (G) -- (H);
            \draw (H) -- (F);
            \draw (G) -- (F);
        
            \draw (G) -- (C);
            \draw (C) -- (F);
            \draw (F) -- (G);
        
            \draw (A) -- (F);
            \draw (F) -- (C);
            \draw (A) -- (C);
        
            \draw (D) -- (I);
        
            \draw (I) -- (H);
        
            \draw (J) -- (H);
        
            \draw (J) -- (I);
        \end{tikzpicture}
        \caption{Simplicial complex $K$.}
        \label{subfig:k0}
    \end{subfigure}
    \hfill
    %K1
    \begin{subfigure}[b]{.45\textwidth}
        %\centering
        \begin{tikzpicture}
            %Start K_1 by removing D.
        
            \node[circle, draw, fill=black, inner sep = 0pt] (A) at (0,0){};
            \node[circle, draw, fill=red, label=above:{$x_1$}, inner sep = 0pt] (B) at (0,2){};
            \node[circle, draw, fill=black, inner sep = 0pt] (C) at (1,1){};
            %remove D.
            \node[circle, draw, fill=black, inner sep = 0pt] (E) at (2,3){};
            \node[circle, draw, fill=black, inner sep = 0pt] (F) at (3,0){};
            \node[circle, draw, fill=black, inner sep = 0pt] (G) at (4,2){};
            \node[circle, draw, fill=black, inner sep = 0pt] (H) at (5,1){};
            \node[circle, draw, fill=black, inner sep = 0pt] (I) at (5,3){};
            \node[circle, draw, fill=black, inner sep = 0pt] (J) at (6,2){};
        
            %Copy all previous edges but exclude the ones removed by Beta(D,K_0).
            \draw (A) -- (B);
            \draw (A) -- (C);
            \draw (B) -- (C);
        
            \draw (C) -- (E);
            \draw (B) -- (E);
            \draw (B) -- (C);
        
            \draw (C) -- (E);
            \draw (C) -- (G);
            \draw (E) -- (G);
        
            \draw (E) -- (I);
            \draw (E) -- (G);
            \draw (G) -- (I);
        
            \draw (G) -- (H);
            \draw (H) -- (F);
            \draw (G) -- (F);
        
            \draw (G) -- (C);
            \draw (C) -- (F);
            \draw (F) -- (G);
        
            \draw (A) -- (F);
            \draw (F) -- (C);
            \draw (A) -- (C);

            \draw (I) -- (H);
        
            \draw (J) -- (H);
        
            \draw (J) -- (I);
        \end{tikzpicture}
        \caption{$K_1$, $x_1$.}
        \label{subfig:k1}
    \end{subfigure}

    \vspace{0.5cm}
    
    %K2
    \begin{subfigure}[b]{0.45\textwidth}
        %\centering
        \begin{tikzpicture}
            %Copy the previous complex, but we strongly collapse using vertex B. 
        
            %Remove vertex B
            \node[circle, draw, fill=black, inner sep = 0pt] (A) at (0,0){};
            %remove B.
            \node[circle, draw, fill=black, inner sep = 0pt] (C) at (1,1){};
            %remove D.
            \node[circle, draw, fill=red, label=above:{$x_2$}, inner sep = 0pt] (E) at (2,3){};
            \node[circle, draw, fill=black, inner sep = 0pt] (F) at (3,0){};
            \node[circle, draw, fill=black, inner sep = 0pt] (G) at (4,2){};
            \node[circle, draw, fill=black, inner sep = 0pt] (H) at (5,1){};
            \node[circle, draw, fill=black, inner sep = 0pt] (I) at (5,3){};
            \node[circle, draw, fill=black, inner sep = 0pt] (J) at (6,2){};
        
            %Copy previous complex but remove the two triangles accordingly, with respect to our definition Beta(B,K_1)
        
            \draw (C) -- (E);
        
            \draw (C) -- (E);
            \draw (C) -- (G);
            \draw (E) -- (G);
        
            \draw (E) -- (I);
            \draw (E) -- (G);
            \draw (G) -- (I);
        
            \draw (G) -- (H);
            \draw (H) -- (F);
            \draw (G) -- (F);
        
            \draw (G) -- (C);
            \draw (C) -- (F);
            \draw (F) -- (G);
        
            \draw (A) -- (F);
            \draw (F) -- (C);
            \draw (A) -- (C);

            \draw (I) -- (H);
        
            \draw (J) -- (H);
        
            \draw (J) -- (I);
            
        \end{tikzpicture}
        \caption{$K_2$, $x_2$.}
        \label{subfig:k2}
    \end{subfigure}
    \hfill
    %K3
    \begin{subfigure}[b]{0.45\textwidth}
        %\centering
        \begin{tikzpicture}
            %Remove E as done for other vertices in the previous simplicial complexes
            \node[circle, draw, fill=black, inner sep = 0pt] (A) at (0,0){};
            %remove B.
            \node[circle, draw, fill=black, inner sep = 0pt] (C) at (1,1){};
            %remove D.
            %remove E.
            \node[circle, draw, fill=black, inner sep = 0pt] (F) at (3,0){};
            \node[circle, draw, fill=black, inner sep = 0pt] (G) at (4,2){};
            \node[circle, draw, fill=black, inner sep = 0pt] (H) at (5,1){};
            \node[circle, draw, fill=black, inner sep = 0pt] (I) at (5,3){};
            \node[circle, draw, fill=red, label=right:{$x_3$}, inner sep = 0pt] (J) at (6,2){};
        
            %Remove all the edges necessary to satisfy the definition of Beta(E,K_2) Using previous complex
            
            \draw (C) -- (G);
        
            \draw (G) -- (I);
        
            \draw (G) -- (H);
            \draw (H) -- (F);
            \draw (G) -- (F);
        
            \draw (G) -- (C);
            \draw (C) -- (F);
            \draw (F) -- (G);
        
            \draw (A) -- (F);
            \draw (F) -- (C);
            \draw (A) -- (C);

            \draw (I) -- (H);
        
            \draw (J) -- (H);
        
            \draw (J) -- (I);
            
        \end{tikzpicture}
        \caption{$K_3$, $x_3$.}
        \label{subfig:k3}
    \end{subfigure}

    \vspace{0.5cm}
    \begin{subfigure}[b]{0.45\textwidth}
        \begin{tikzpicture}
            %Now we showcase strong collapsing works for one triangle by removing Beta(J,K_3)
        
            %From previous complex, remove node J.
            \node[circle, draw, fill=red, label=left:{$x_4$}, inner sep = 0pt] (A) at (0,0){};
            %remove B.
            \node[circle, draw, fill=black, inner sep = 0pt] (C) at (1,1){};
            %remove D.
            %remove E.
            \node[circle, draw, fill=black, inner sep = 0pt] (F) at (3,0){};
            \node[circle, draw, fill=black, inner sep = 0pt] (G) at (4,2){};
            \node[circle, draw, fill=black, inner sep = 0pt] (H) at (5,1){};
            \node[circle, draw, fill=black, inner sep = 0pt] (I) at (5,3){};
            %remove J.
        
            %From previous complex, remove the simplicies described by Beta(J,K_3)
            
            \draw (C) -- (G);
        
            \draw (G) -- (I);
        
            \draw (G) -- (H);
            \draw (H) -- (F);
            \draw (G) -- (F);
        
            \draw (G) -- (C);
            \draw (C) -- (F);
            \draw (F) -- (G);
        
            \draw (A) -- (F);
            \draw (F) -- (C);
            \draw (A) -- (C);
        
            \draw (H) -- (I);
        \end{tikzpicture}
        \caption{$K_4$, $x_4$.}
        \label{subfig:k4}
    \end{subfigure}
    \hfill
    \begin{subfigure}[b]{0.45\textwidth}
        \begin{tikzpicture}
            %We get rid of one triangle again this time, we remove A. Using the previous complex, we get
            %remove A.
            %remove B.
            \node[circle, draw, fill=black, inner sep = 0pt] (C) at (1,1){};
            %remove D.
            %remove E.
            \node[circle, draw, fill=black, inner sep = 0pt] (F) at (3,0){};
            \node[circle, draw, fill=black, inner sep = 0pt] (G) at (4,2){};
            \node[circle, draw, fill=red, label=right:{$x_5$}, inner sep = 0pt] (H) at (5,1){};
            \node[circle, draw, fill=black, inner sep = 0pt] (I) at (5,3){};
            %remove J.
        
            %Using the previous complex, we remove the edges associated to Beta(A,K_4)
               
            \draw (C) -- (G);
        
            \draw (G) -- (I);
        
            \draw (G) -- (H);
            \draw (H) -- (F);
            \draw (G) -- (F);
        
            \draw (G) -- (C);
            \draw (C) -- (F);
            \draw (F) -- (G);
        
            \draw (F) -- (C);
        
            \draw (H) -- (I);
            
        \end{tikzpicture}
        \caption{$K_5$, $x_5$.}
        \label{subfig:k5}
    \end{subfigure}

    \vspace{0.5cm}

    \begin{subfigure}[b]{0.45\textwidth}
        \begin{tikzpicture}
            \centering
            %Remove H and we will leave K_N as a triangle for now. Notice that removing H would require us to remove I (by definition of Beta(H,K_5). Using the previous complex, we get
        
             %remove A.
            %remove B.
            \node[circle, draw, fill=red,label = above:{$x_6$}, inner sep = 0pt] (C) at (1,1){};
            %remove D.
            %remove E.
            \node[circle, draw, fill=black, inner sep = 0pt] (F) at (3,0){};
            \node[circle, draw, fill=black, inner sep = 0pt] (G) at (4,2){};
            %remove H.
            %remove I.
            %remove J.
        
            %Using Beta(H,K_5), we make use of the previous complex to remove the relevant edges.
        
            \draw (G) -- (C);
            \draw (F) -- (G);
        
            \draw (F) -- (C);
        \end{tikzpicture}
        \caption{$K_6$ is the resulting simplicial complex after strongly collapsing $K$ to a triangle.}
        \label{subfig:k6}
    \end{subfigure}
    \caption{Example of strongly collapsing simplicial complex $K$ into a triangle.}
    \label{fig:strongcollapse}
\end{figure}

\newpage
We will need a preliminary lemma which in turn needs the definition of a hat function.  Let $\Th$ be a
mesh and let $x$ be a vertex  of the mesh. We let $\lambda_x \in V_h^1(\Th)$ be the piecewise linear function that is continuous on the domain formed by $\Th$  such that 
$$
    \lambda_x(y)=
    \begin{cases}
        1\qquad\text{if } y =x, y \in \Delta_0(\Th) \\
        0\qquad\text{if } y \neq x, y \in \Delta_0(\Th). 
    \end{cases}\label{lambdaDef}
    $$

\begin{lemma}\label{lemma101}
Consider a vertex $x$ and suppose that $x \in \Delta_0(T_1), x \in \Delta_0(T_2)$ for two triangles $T_1$ and $T_2$ with a common edge $e \in \Delta_1(T_1), e \in \Delta_1(T_2)$. Let $e_i \in \Delta_1(T_i)$ opposite to the vertex $x$.  We assume that $e_1$ and $e_2$ do not lie in the same line (see fig. \ref{fig:colinearedges}). For any piecewise constant function $p$ defined on $T_1 \cup T_2$ with $p|_{T_i}$ being constant there exists a constant vector $\vec{c}$ such that  

\begin{equation*}
    \dive (\vec{c} \lambda_x)= p  \quad \text{ on }  T_1 \cup T_2.
\end{equation*}
where $\lambda_x$ is the hat function associated to $x$. 
\end{lemma}

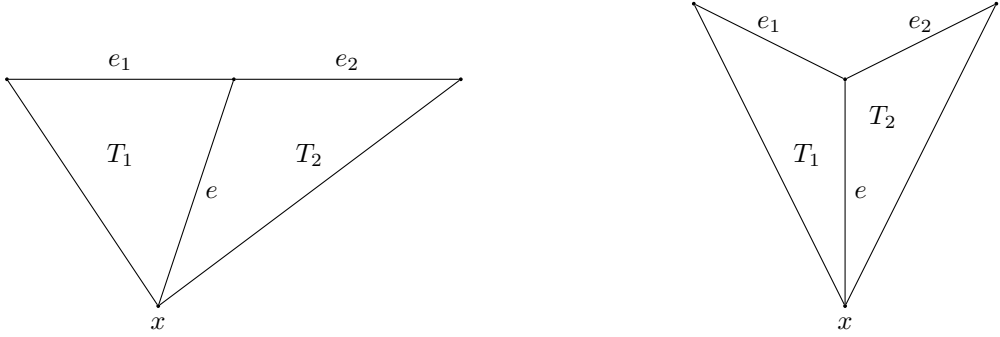
\begin{figure}[h]
    \centering
    \begin{subfigure}[b]{0.45\textwidth}
        \begin{tikzpicture}
            \node[circle, draw, fill=black, inner sep = 0pt] (A) at (-2,4){};
            \node[circle, draw, fill=black, inner sep = 0pt] (B) at (1,4){};
            \node[circle, draw, fill=black, inner sep = 0pt] (C) at (4,4){};
            \node[circle, draw, fill=black, inner sep = 0pt, label=below : {$x$}] (D) at (0,1){};

            \draw (A) -- (B) node[midway, above] {$e_1$};
            \draw (B) -- (C) node[midway, above]{$e_2$};
            \draw (A) -- (D);
            \draw (D) -- (C);
            \draw (D) -- (B) node[midway, right]{$e$};

            %The labels of T_1 and T_2 for the first complex.
            \node at (-.5,3){$T_1$};
            \node at (2,3){$T_2$};
        \end{tikzpicture}
    \end{subfigure}
    \hfill
    \begin{subfigure}[b]{0.45\textwidth}
        \begin{tikzpicture}
            \node[circle, draw, fill=black, inner sep = 0pt] (E) at (5,4){};
            \node[circle, draw, fill=black, inner sep = 0pt] (F) at (7,3){};
            \node[circle, draw, fill=black, inner sep = 0pt] (G) at (9,4){};
            \node[circle, draw, fill=black, inner sep = 0pt, label = below : {$x$}] (H) at (7,0){};

            \draw (E) -- (F) node[midway, above]{$e_1$};
            \draw (F) -- (G) node[midway, above]{$e_2$};
            \draw (E) -- (H);
            \draw (G) -- (H) ;
            \draw (H) -- (F) node[midway, right]{$e$};

            %These are the labels of both triangles for the complex on the right.
            \node at (6.5, 2) {$T_1$};
            \node at (7.5, 2.5) {$T_2$};
        \end{tikzpicture}
            
    \end{subfigure}
    
    \caption{The lemma does not allow the left triangulation}
    \label{fig:colinearedges}
\end{figure}

\begin{proof}
    Let 
    \begin{equation*}
        p=
        \begin{cases}
            \ p_1 & \text{on }T_1\\
            \ p_2 &\text{on }T_2
        \end{cases},
    \end{equation*}
 
    \begin{alignat*}{2}
        &h_1 = \text{dist}(x, l(e_1)) &\qquad\qquad 
        &h_2 = \text{dist}(x, l(e_2)),\\
        \intertext{ and}
        &\nabla \lambda_x|_{T_1} = -\vec{n_1}h_1 &\qquad\qquad &\nabla\lambda_x|_{T_2} = -\vec{n_2}h_2
    \end{alignat*}
    where $l(e_i)$ is the line on which $e_i$ lies on, $h_i$ is the shortest distance between $x$ and the line containing $e_i$, and $\vec{n_i}$ is the vector normal to the edge $e_i$. 
    Then we need 
    \begin{equation*}
        \vec{c}\cdot (-h_i\vec{n_i}) = p_i 
    \end{equation*}

    Using the notation $ \vec{c} = \begin{bmatrix} c_1 \\ c_2 \end{bmatrix}$ and $\vec{n_i} = \begin{bmatrix} n_{i1} \\ n_{i2} \end{bmatrix}$  and defining
     \begin{equation*}
        A: =
        \begin{bmatrix}
            h_1n_{11} & h_1n_{12}\\
            h_2n_{21} & h_2n_{22}
        \end{bmatrix}
    \end{equation*}
Then we see that we need to solve $-A\vec{c} =\vec{P}$ where $ \vec{P} = \begin{bmatrix} p_1 \\ p_2 \end{bmatrix}$. There is a solution as long as $A$ is invertible or equivalently if the determinant of $A$ does not vanish.  But the determinant is $h_1 h_2(n_{11} n_{22}-n_{12}n_{21})$. This is not zero as long as $n_1$ is not a multiple of $n_2$ which occurs when $e_1$ and $e_2$ do not lie in the same line.  
\end{proof}
Similarly, we can prove a more simple result. 
\begin{lemma}\label{lemma102}
Consider a vertex $x$ and suppose that $x \in \Delta_0(T)$ for a triangle $T$. If $p$ is constant on $T$ there exists a constant vector $\vec{c}$ such that 
\begin{equation*}
    \dive (\vec{c} \lambda_x)= p  \quad \text{ on }  T.
\end{equation*} 
\end{lemma}

Now we can relate strong collapsiblity to Strang's conjecture.
\begin{theorem}Let $\Th$ be triangulation of a two dimensional domain $\Omega$. Let $K$ be its associated simplicial complex. \\

1) Suppose that $K$ is strongly collapsible.  \\

2) (non-collinear condition) In addition, suppose for each $x_i$ (in Definition \ref{strong})  that is a vertex of two triangles  $T_1,T_2 \in K_i$ then the edges of $T_1, T_2$  opposite to $x_i$ do not lie on the same line. \\

Then, Strang's Conjecture holds on $\Th$ in the case $k=2$ (i.e $\dim\sk=\mathcal{S}_h^k$). 
\end{theorem}

\begin{proof}
    For every $p\in\dzero$ we will show that there exists a $\vec{v} \in \vone$ such that $\dive \vec{v}=p$.
     Since $K$ is strongly collapsible to a triangle (see fig. 2), we know that there exist a sequence of vertices $x_0,x_1,\dots,x_{N}$ that correspond with each complex $K_0,K_1,\dots, K_N$. 

     Suppose that we have constructed a $\vec{v}_{i+1} \in \vone$ such that 
     \begin{equation*}
         \dive \vec{v}_{i+1}= p \qquad \text{ on } \Re(K_{i+1}).
     \end{equation*}
     Here we used the notation 
     \begin{alignat*}{1}
          \Re(K_{i+1})=\bigcup_{T \in \Delta_2(K_{i+1})} T
     \end{alignat*}

     Let $p_i=  p- \dive \vec{v}_{i+1} \in \dzero$ and we note that $p_i$ vanishes on $\Re(K_{i+1})$. Let $B_i= \Re(K_{i}) \backslash \Re(K_{i+1})$ and by our hypothesis $B_i$ is the region with either one triangle or two triangles and of course the vertex $x_i$.   Let $\vec{w}_i= \vec{c}_i \lambda_{x_i}$  for some constant vector $\vec{c}_i$ which we choose in a moment. We note that $\vec{w}_i$ vanishes on $\Re(K_{i+1})$. By Lemmas \ref{lemma101}, \ref{lemma102} we have 
    \begin{equation*}
         \dive \vec{w}_i = p_i \qquad \text{ on } B_i.
     \end{equation*}

     We define $\vec{v}_i \in \vone$ by
     \begin{alignat*}{1}
      \vec{v}_i:=\vec{w}_i+ \vec{v}_{i+1}.   
     \end{alignat*}

     We have
     \begin{alignat*}{2}
         \dive \vec{v}_i=& \dive \vec{w}_i+  \dive \vec{v}_{i+1} \quad &&  \\
         =& p_i + \dive \vec{v}_{i+1} \quad && \\
         =& p  \qquad &&  \text{on }  B_i.
     \end{alignat*}
     On the other hand,
     \begin{alignat*}{2}
         \dive \vec{v}_i=& \dive \vec{w}_i+  \dive \vec{v}_{i+1} \quad &&  \\
         =& \dive \vec{v}_{i+1} \quad &&  \\
         =& p \qquad && \text{on }  \Re(K_{i+1})
     \end{alignat*}
     Hence, we have shown 

     \begin{equation*}
         \dive \vec{v}_{i}= p \qquad \text{ on } \Re(K_{i}).
     \end{equation*}
     If we continue this process to produce $\vec{v}_N, \vec{v}_{N-1}, \ldots, \vec{v}_0$ then we let $\vec{v}:=\vec{v}_0$ is the desired vector field.
    \end{proof}

\section{Acknowledgements}
 We would like to thank the entire MSRI-UP 2026 staff. We cannot say thanks enough to the faculty members Alexander Diaz-Lopez, Johnny Guzm\'an and Maurice Fabien. We extend our deepest and sincere appreciation to our graduate mentors Parneet Gill and Elizabeth Rubio for their moral and academic support throughout the research.  Many thanks to our primary research mentor Johnny Guzm\'an. We would also like to give special thanks to Maurice Fabien for collaborating with us to write the code to check Strang's conjecture on a given mesh, which is available here: \hyperlink{https://github.com/armro/CheckStrangsConjecture}{https://github.com/armro/CheckStrangsConjecture}. 

 This research was conducted during the MSRI-UP 2026 summer program. We would like to thank SL Math for their hospitality.

\newpage
\bibliographystyle{abbrv}
\bibliography{./references}

\end{document}